\documentclass{amsart}
\usepackage {amsmath, amscd}
\usepackage {amssymb}
\usepackage {epsfig}
\usepackage {bm}
\usepackage {indentfirst} 
\usepackage{color}

\numberwithin{equation}{section}

\def\<{\langle}
\def\>{\rangle}

\def\BB{{\mathcal B}}
\def\CC{{\mathcal C}}

\def\TT{{\mathcal T}}

\def\bbR{\mathbb{R}}
\def\bbC{\mathbb{C}}
\def\bbZ{\mathbb{Z}}
\def\bbN{\mathbb{N}}

\def\bbD{\mathbb{D}}
\def\bbT{\mathbb{T}}

\def\1{\mathbf{1}}

\newcommand{\Tr}{\mathop{\rm Tr}}
\newcommand{\Arg}{\mathop{\rm Arg}}

\newtheorem{lemma}{Lemma}[section]
\newtheorem{theorem}[lemma]{Theorem}
\newtheorem{corollary}[lemma]{Corollary}

\theoremstyle{definition}

\newtheorem{example}[lemma]{Example}

\title{A Szeg\"o type theorem for truncated Toeplitz operators}

\author{Elizabeth Strouse}
\address{ Universit\'e de Bordeaux, Institut de Math\'ematiques de Bordeaux UMR 5251,
	351, cours de la Lib\'eration,
	F-33405 Talence cedex, France}
\email{Elizabeth.Strouse@math.u-bordeaux1.fr}

\author{Dan Timotin}
\address{Institute of Mathematics of the Romanian Academy, P.O. Box 1-764, Bucharest 
	014700, Romania}
\email{Dan.Timotin@imar.ro}

\author{Mohamed Zarrabi}
\address{ Universit\'e de Bordeaux, Institut de Math\'ematiques de Bordeaux UMR 5251,
	351, cours de la Lib\'eration,
	F-33405 Talence cedex, France}
\email{Mohamed.Zarrabi@math.u-bordeaux1.fr}


\begin{document}

\begin{abstract}
	Truncated Toeplitz operators are compressions of multiplication operators on $L^2$ to model spaces (that is, subspaces of $H^2$ which are invariant with respect to the backward shift). For this class of operators we prove certain Szeg\"o type theorems concerning the asymptotics of their compressions to an increasing chain of finite dimensional model spaces.
\end{abstract}

\maketitle
The Toeplitz operators are compressions of multiplication operators on the space $L^2(\bbT)$  to the Hardy space $H^2$; the multiplier is  called the symbol of the operator. With respect to the standard exponential basis, their matrices are constant along diagonals; if we truncate such a matrix considering only its upper left finite corner, we obtain classical Toeplitz matrices.

It does not come as a surprise that there are connections between the asymptotics of these Toeplitz matrices and the whole Toeplitz operator, or its symbol.
A central result is Szeg\"o's strong limit theorem and its variants (see, for instance,~\cite{BS} and the references within), which deal with the asymptotics of the eigenvalues of the Toeplitz matrix.

On the other hand, certain generalizations of Toeplitz matrices have attracted a great deal of attention in the last decade, namely compressions of multiplication operators to subspaces of the Hardy space which are invariant under the backward shift. These  ``model spaces'' are of the form $H^2\ominus uH^2$ with $u$ an inner function, and the compressions are called truncated Toeplitz operators. They have been formally introduced in~\cite{Sa}; see~\cite{GR} for a more recent survey. Although classical Toeplitz matrices have often been a starting point for investigating truncated Toeplitz operators, the latter may exhibit surprising properties.

It thus seems  natural to see whether an analogue of Szeg\"o's strong limit theorem can be obtained in this more general context. Viewed as truncated Toeplitz operators, the Toeplitz matrices act on model spaces corresponding to the inner functions $u(z)=z^n$, and Szeg\"o's theorem is about the asymptotical situation when  $n\to\infty$. The natural generalization is then to consider a sequence of zeros $(\lambda_j)$ in $\bbD$, and to let the truncations act on the model space corresponding to the finite Blaschke product associated to $\lambda_j$, $1\le j\le n$.

Such a result has been obtained in~\cite{B}; it deals with the asymptotics of the determinant of a truncated Toeplitz operator. Let us note that in the case of classical Toeplitz operators and matrices one has different variants of Szeg\"o's Theorem, either in terms of the determinant of the truncation, or in terms of the trace, and one can pass from one to the other. However, this is no longer true in our generalized context, where the two different classes of results do not have a visible connection.

The purpose of this paper is to find an analogue of the trace type Szeg\"o Theorem. We manage to obtain a complete result in the case when $(\lambda_j)$ is not a Blaschke sequence. The Blaschke case seems to be less prone to an elegant solution, and we have only partial conclusions.

The technique we use is inspired by one of the approaches to the Szeg\"o Theorem that is based on approximation by circulants (see for instance,~\cite{Gr}). In our case we use the analogue of circulants for  truncated Toeplitz operators, namely elements in the so-called 
Sedlock algebras~\cite{Se}.

The plan of the paper is the following. After a rather extensive preliminary section that introduces the basic notions, we discuss the special case of finite dimensional model spaces. We introduce then the Sedlock algebras in Section~\ref{se:circulants}, and prove a result important in its own right, Theorem~\ref{th:circulants as functions generalized}, which gives an alternate identification of these algebras. Sedlock algebras on finite dimensional model spaces are briefly discussed  in Section~\ref{se:sedlock and finite blaschke}, after which Section~\ref{se:approx by circulants} develops the approximation technique based on them. The main result, Theorem~\ref{th:general convergence for non-Blaschke}, is proved in Section~\ref{se:main theorem}. The last section discusses through some examples the problems that appear when considering Blaschke sequences.

\section{Preliminaries}

\subsection{Model spaces}
Let $H^2$ be the Hardy space of square integrable functions on the circle with negative Fourier coefficients equal to 0.
We recall that a \emph{model} space is a subspace of  $H^2$ which is invariant for the backwards shift, and that every such space is of the form
$K_B =H^2\ominus BH^2$ where $B$ is an inner function, i.e., an element of $H^2$ of modulus 1 almost everywhere. We write $P_B$ for orthogonal projection from $L^2 $ onto $K_B.$

We will often make the supplementary assumption that $B(0) = 0 $.  In this case $\hat{B}$ is defined by
$B = z\hat{B}.$

The reproducing kernels for $K_B$ at the points $\lambda \in \bbD$ are the functions
\[
k^B_{\lambda}(z)=\frac{1-\overline{B(\lambda)}B(z)}{1-\overline{\lambda}z}, \qquad z\in\bbD.
\]
When  $B$ has an angular derivative in the sense of Caratheodory at a point in $\zeta\in \bbT$,  that is, when the
nontangential limit of $B$ and $B^{'}$ exist in $\zeta$ with the limit of $B$ in $\zeta$ of module 1, then all functions in $K_B$ have a radial limit in $\zeta$, and the corresponding reproducing kernel $k^B_\zeta$ belongs to $K_B$. We have
\[
k^B_{\zeta}(z)=\frac{1-\overline{B(\zeta)}B(z)}{1-\overline{\zeta}z}={\zeta\overline{B(\zeta)}}\frac{B(\zeta)-B(z)}{\zeta-z} , \qquad z\in\bbD.
\]

We also use the fact that
\begin{equation}\label{eq:formula for norm of k_zeta}
\|k^B_\zeta\|^2= k^B_\zeta (\zeta )  =\lim_{z\to\zeta}\frac{1-{\overline{B(\zeta)}}B(z)}{1-\overline{\zeta}z} =\frac{\zeta B'(\zeta)}{B(\zeta)} =
|B'(\zeta)|.
\end{equation}
(The last equality follows from the fact that all of the above terms are positive real numbers.)

\subsection{Truncated Toeplitz operators}\label{sse:TTO basics}

	If $\phi \in L^2$, then the map $f\mapsto P_B  \phi f$ is linear from the dense subspace $H^\infty\cap K_B$ to $K_B$. When this map is bounded, it can be extended to a bounded linear operator on the whole $K_B$, that will be denoted by $T_B[\phi]=P_B M_\phi |K_B$ and called a \emph{truncated Toeplitz operator}.
For $\phi(z)=z$ one obtains the \emph{compressed shift}, which we will denote by $S_B$.
The closed linear space of all bounded truncated Toeplitz operators on $K_B$ will be denoted by $\TT_B$.

The function $\phi$ is called the \emph{symbol} of the operator.
It is known ~\cite{Sa} that any operator in $\TT_B$ has a   symbol $\phi$ in $K_B+\overline{K_B}$, which is unique in case $B(0)=0$. Such a symbol is called \emph{standard}. Also, if
 $B(0)=0$ we see that $1 \in K_B$, so that the orthogonal complement of the constants in $K_B$ equals
$zK_{\hat{B}}$. Thus:
$$K_B+\overline{K_B} = K_B \oplus {\overline{zK_{\hat{B}}}}$$
and standard symbols can be uniquely written in the form $\phi_+ + \overline{\phi_-}$, with $\phi_+\in K_B$ and $\phi_-\in zK_{\hat B}$. Moreover, using the
fact that for any inner function $U$ the map $h\mapsto \bar z\bar h U$ is an involution on $K_U$ one sees that
\[
\{\overline{\phi_-}: \phi_-\in zK_{\hat B}\}=
\{\bar B\psi_-: \psi_-\in zK_{\hat B}\}
\]
and so we will write our standard symbol (uniquely) as $\psi_+ +\bar B \psi_-$ with
$\psi_+\in K_B$ and $\psi_-\in zK_{\hat B}$.

For $\alpha\in\bbT$ one defines $S_B^\alpha=S_B+\alpha(1\otimes\hat B)$. One can check easily that the operators $S_B^\alpha$ (called  \emph{modified compressed shifts}) are unitary.  It is proved in~\cite{Sa} that the modified compressed shifts belong to $\TT_B$.

Suppose now that $a\in\bbD$, and define $b_a(z)=\frac{z-a}{1-\bar a z}$; $b_a(z)$ is an automorphism of the unit disk. We set $B^a=B\circ b_a$. The following result, that we will have the opportunity to use below, is Proposition~4.1 of~\cite{CGRW}.

\begin{lemma}\label{le:CGRW}
	The formula
	\begin{equation}\label{eq:definition of U_a}
	U_a(f)= \sqrt{b'_a}\cdot f\circ b_a = \frac{\sqrt{1-|a|^2}}{1-\bar a z} f\circ b_a
	\end{equation}
	defines a unitary operator
	 $U_a$ from $K_B$ to $K_{B^a}$, and for any  symbol $\phi$ we have
	\[
	U_a T_B[\phi]U_a^*=T_{B^a} [\phi\circ b_a].
	\]
\end{lemma}

\subsection{Clark Measures}\label{sse:clark measures}

The following  measures have been introduced by Clark~\cite{Clark}.
For  $\alpha \in \bbT$, the function $\frac{\alpha+B(z)}{\alpha-B(z)}$ has positive real part, and therefore, by Herglotz theorem, there exists a finite positive measure $\mu_\alpha^B$ on $\bbT$ such that
\begin{equation}\label{eq:definition of mualpha}
\Re \left(\frac{\alpha+B(z)}{\alpha-B(z)}\right)=\int_\bbT P_{r,t}(\zeta) d\mu_\alpha^B (\zeta),
\end{equation}
where $P_{r,t}(\zeta)=\frac{1-r^2}{\vert \zeta - re^{it}\vert^2}$ is the Poisson kernel.
The measure is singular, and in case $B(0)=0$  we have $\|\mu_\alpha^B\|=1$. A result of Alexandrov states that for every continuous function $f$
\begin{equation}\label{eq:aleksandrov}
\int_\bbT f dm=\int_\bbT \left(\int_\bbT f d\mu_\alpha\right) dm(\alpha),
\end{equation}
where $m$ is the normalized Lebesgue measure.
We will often write $\mu_\alpha$ instead of $\mu_\alpha^B$ when there is a single inner function~$B$ involved.

We note that the Clark measures $\mu_\alpha^b$ are defined in the same way for any function $b$ in the unit ball of $H^\infty$. In particular for the zero function $b=0$, for every $\alpha$, $\mu_\alpha^b$ is the normalized Lebesgue measure.

The next theorem combines results from~\cite{Clark} and ~\cite{Po}.

\begin{theorem}\label{th:clarkpoltoratski} Suppose that $B$ is an inner function, $\alpha\in\bbT$, and $\mu_\alpha$ is defined by~\eqref{eq:definition of mualpha}. Then
	any function  $f\in K_B$ has a radial limit $f^*$ almost everywhere with respect to $\mu_\alpha$,  the operator $V:K_u\to L^2(\mu_\alpha)$ defined by $Vf=f^*$ is unitary, and $VS_B^\alpha=M_z V$.
\end{theorem}

Suppose now that $\lambda_j\in \bbD$, $j\ge 1$. Define, for $n\ge 1$, 
\begin{equation}\label{eq:definition of Bn}
B_n=\prod_{j=1}^{n} (-\frac{|\lambda_j|}{\lambda_j}b_{\lambda_j})
\end{equation}
(By convention, in case $\lambda_j=0$, the corresponding factor will be $-b_0=-z$.)
The following lemma is well known;  we provide a proof for completeness.

\begin{lemma}\label{le:convergence of clark measures}
	Fix $\alpha\in\bbT$, and consider for each $n$ the measure $\mu_\alpha^{B_n}$.
	\begin{itemize}
		\item[(i)] If $\sum (1-|\lambda_j|)=\infty$,  then $\mu_\alpha^{B_n}$ converges in the weak star topology to  normalized Lebesgue measure.
		
		\item[(ii)] If $\sum (1-|\lambda_j|)<\infty$,  then $\mu_\alpha^{B_n}$ converges in the weak star topology to $\mu_\alpha^B$, where $B=\prod_{j=1}^{\infty} (-\frac{|\lambda_j|}{\lambda_j}b_{\lambda_j})$ is the infinite Blaschke product with zeros $\lambda_j$.
	\end{itemize}
	\end{lemma}

\begin{proof}
The sequence  $(B_n)_n$ converges uniformly on every compact set of $\bbD$ to  the function $B$, where $B$ is the zero function in case (i) and  the infinite Blaschke product with zeros $\lambda_j$ in case (ii).  For $0\leq r<1$ and $t\in\bbR$, the Poisson kernel is defined by $P_{re^{it}} (\zeta)=\frac{1-r^2}{\vert \zeta - re^{it}\vert^2}=\sum_{n\in\bbZ} r^{\vert n\vert} e^{-int}\zeta^n$. We have
 $$
 \int_\bbT P_{re^{it}} (\zeta)   d\mu_\alpha^{B_n} (\zeta)= \Re \left(\frac{\alpha+B_n(re^{it})}{\alpha-B_n(re^{it})}\right)\to
\Re \left(\frac{\alpha+B(re^{it})}{\alpha-B(re^{it})}\right)=\int_\bbT  P_{re^{it}} (\zeta) d\mu_\alpha^{B} (\zeta).
$$
 Since $span\{P_{re^{it}};\ 0\leq r<1,\ t\in\bbR\}$ is dense in the space  of continuous functions, we get that $(\mu_\alpha^{B_n})_n$ converges in the weak star topology to
$\mu_\alpha^B$. Now the proof is finished since $\mu_\alpha^B$ is the normalized Lebesgue measure when $B\equiv 0$.
\end{proof}

The behavior of the Clark measures with respect to change of variable by an automorphism of the disc is given by the next lemma, whose proof is a simple change of variable in~\eqref{eq:definition of mualpha}.

\begin{lemma}\label{le:change of Clark by automorphism}
	With the above notations, we have $\mu^{B^a}_\alpha= (b_{-a})_* (|b'_{-a}|\mu_\alpha^B)$.
\end{lemma}

\subsection{The Cauchy transform}
For the next  facts about Cauchy transforms we refer to~\cite{CMR}. If $\mu$ is a Borel measure on $\bbT$, its Cauchy transform $K_\mu$ is the analytic function on $\bbD$ defined by
\[
K_\mu(z)=\int_{\bbT} \frac{1}{1-\bar{\zeta}z} \, d\mu(\zeta).
\]
The following lemma summarizes the properties of the Cauchy transform that we need.

\begin{lemma}\label{le:cauchy transform}
	Suppose $\mu$ is a Borel measure on $\bbT$. Then:
	\begin{enumerate}
		\item $K_\mu\in H^p$ for all $0<p<1$.
		
		\item $K_\mu\equiv 0$ if and only if $d\mu= \bar{\phi} \,dm$ for some $\phi\in H^1_0$,
where 	$H^1_0=\{f\in H^1,\ f(0)=0\}$.
	\end{enumerate}
\end{lemma}

Suppose $\Lambda=(\lambda_j)_{j\ge 1}$, and each $\lambda_j$  is repeated $m_j$ times, with $m_j$ a finite integer or infinite. We will denote by $\mathfrak{L}_\Lambda$ the linear span of $z^m k^{m+1}_{\lambda_j}$ and $\bar z^m \bar k^{m+1}_{\lambda_j}$, with $j\ge 1$ and $0\le m <m_j$.
The next lemma  is probably known, but  we have not found an appropriate reference.

\begin{lemma}\label{le:density of kernels}
 Suppose $\sum_{j=1}^{\infty}(1-|\lambda_j|)=\infty $. Then $\mathfrak{L}_\Lambda$ is dense in $C(\bbT)$.
\end{lemma}

\begin{proof}
	Suppose $\mu$ is a Borel measure on $\bbT$ such that
$$\int_\bbT z^mk^{m+1}_{\lambda_j }(z) \, d\mu=\int_\bbT \bar z^m\bar k^{m+1}_{\lambda_j }(z) \, d\mu=0 $$ for all $j\ge 1$ and $0\le m <m_j$. By replacing $\mu$, if necessary, with $\frac{1}{2} (\mu+\bar{\mu})$ and $\frac{1}{2i} (\mu-\bar{\mu})$, we may assume that $\mu$ is real. We have then $K^{(m)}_\mu(\lambda_j)=0$ for all $j\ge 1$ and $0\le m <m_j$. Since $K_\mu\in H^p$ for $0<p<1$, if it is not identically zero its zeros would have to satisfy the Blaschke condition. So $K_\mu\equiv 0$, and then by Lemma~\ref{le:cauchy transform} $\mu=\bar\phi\, dm$ for some $\phi\in H^1_0$. Since $\mu$ is real, $\phi$ is real almost everywhere. Therefore $\phi\equiv 0$ and  $\mu\equiv0$, which finishes the proof.
\end{proof}

If $B_n$ is defined by~(\ref{eq:definition of Bn}), then each  of the functions $z^m k^{m+1}_{\lambda_j}$ belongs to $B_n$ for sufficiently large $n$. This remark yields the next corollary.

\begin{corollary}\label{co:desity fo kernels}
Suppose $\sum_{j=1}^{\infty}(1-|\lambda_j|)=\infty $, and $B_n$ is defined by ~(\ref{eq:definition of Bn}). Then $\bigvee_{n=1}^\infty (K_{B_n}+\overline{K_{B_n}})$ is dense in $C(\bbT)$.
\end{corollary}

\section{Finite dimensional model spaces}\label{se:finite dim}

Let us suppose now that $u=B$ is a finite Blaschke product of order $N$. In this case  $K_B$ is of dimension $N$ and contains only rational functions with poles outside the closed unit disc (more precisely, in the reciprocals of the zeros of $B$); we have  $\dim K_B=N$ and $\dim \TT_B=2N+1$.  The measure $\mu_\alpha^B$ is concentrated on the roots $\zeta_k^\alpha$ ($k=1,\dots, N$) of the equation $B(\zeta)=\alpha$, and is given by the formula
\begin{equation}\label{eq:definition of Clark measure}
\mu^B_\alpha =\sum_{k=1}^N \frac{1}{|B'(\zeta^\alpha_k)|}\delta_{\zeta^\alpha_k}.
\end{equation}
The functions in $K_B$ as well as the standard symbols of TTOs have well defined values with respect to $\mu_\alpha^B$.

We will have  the opportunity to use the next lemma, which completes~\eqref{eq:formula for norm of k_zeta}.

\begin{lemma}\label{le:blaschke argument derivative}
	Suppose $B=\prod_{j=1}^{n}b_{\lambda_j}$. Then we have:
	\begin{itemize}
		\item[(i)]	\begin{equation*}\label{eq:argument derivative}
		|B'(e^{it})|=	e^{it}\frac{B'(e^{it})}{B(e^{it})}
		= \sum_{j=1}^{n} \frac{1-|\lambda_j|^2}{|e^{it}-\lambda_j|^2}
		=\frac{d}{dt} \Arg B(e^{it}).
		\end{equation*}
		
		\item[(ii)] For all $\zeta\in\bbT$ we have $|B'(\zeta)|\ge \frac{1}{2}\sum_{j=1}^n(1-|\lambda_j|)$.
	\end{itemize}

\end{lemma}

\begin{proof}
	(i) The second equality is just a computation, using the formula for $B$. Taking absolute values, we deduce the first equality.
	
	For the third equality, note that
	\[
	\Arg B(e^{it})= \Im \log (B(e^{it}))=
	\Im\left( \sum_{j=1}^{n} \log b_{\lambda_j} (e^{it}) \right).
	\]
	Differentiating with respect to $t$, we obtain
	\[
	\frac{d}{dt} \Arg B(e^{it})	=\Im \left( \sum_{j=1}^{n} i \frac{1-|\lambda_j|^2}{|e^{it}-\lambda_j|^2}\right)=\sum_{j=1}^{n} \frac{1-|\lambda_j|^2}{|e^{it}-\lambda_j|^2}
	\]
	as required.
	
	Part (ii) is a consequence of (i), once we note that for all $\zeta\in\bbT$ and $\lambda\in\bbD$ we have $1-|\lambda|\le |\zeta-\lambda|\le 1+|\lambda|$, and therefore
	\[
	\frac{1-|\lambda|}{1+|\lambda|}\le \frac{1-|\lambda|^2}{|\zeta-\lambda|^2}\le \frac{1+|\lambda|}{1-|\lambda|}. \qedhere
	\]
\end{proof}

For each $\alpha\in\bbT$ the reproducing kernels $k_{\zeta_k^\alpha}^B$, $k=1,\dots, N$, form an orthogonal base of $K_B$.
 The operator $S_B^\alpha$ has as eigenvalues $\zeta_k^\alpha$ ($k=1,\dots, N$),  with corresponding eigenvectors  $k_{\zeta_k^\alpha}^B$.

Consider a measure $\nu=\sum_{k=1}^N \nu_k \delta_{\zeta_k^\alpha}$ concentrated on the points $\zeta_k^\alpha$
and define the operator $D_\nu$ by
\begin{equation}\label{eq:definition of D_nu}
D_\nu= \sum_{k=1}^N \nu_k\left(
\frac{k^B_{\zeta_k^\alpha}}{\|k^B_{\zeta_k^\alpha}\|} \otimes
\frac{k^B_{\zeta_k^\alpha}}{\|k^B_{\zeta_k^\alpha}\|}  \right).
\end{equation}
Then $D_\nu$ is a function of $S_B^\alpha$, and therefore belongs to $\TT_B$. In particular, we will denote
$\Delta^\alpha_B=D_{\mu^B_\alpha}$; thus
\begin{equation}\label{eq:definition of Delta}
\Delta_B^\alpha:= \sum_{k=1}^N \frac{1}{|B'(\zeta_k^\alpha)|}\left(
\frac{k^B_{\zeta_k^\alpha}}{\|k^B_{\zeta_k^\alpha}\|} \otimes
\frac{k^B_{\zeta_k^\alpha}}{\|k^B_{\zeta_k^\alpha}\|}  \right),
\end{equation}

 The next lemma gives more information about the unitary operators that have appeared in Lemma~\ref{le:CGRW}.

\begin{lemma}\label{le:action of U_a on kernels and measures}
	Suppose $a\in\bbD$, $\alpha\in\bbT$, and  the operator $U_a:K_B\to K_{B^a}$ is defined by~\eqref{eq:definition of U_a}.
	\begin{itemize}
		\item[(i)] For each $k=1, \dots, N$,
		\[
	U_a(k^B_{\zeta_k^\alpha}) = \sqrt{b'_{-a}(\zeta_k^\alpha)}k^{B^a}_{\eta^\alpha_k}.
	\]
	where $\eta^\alpha_k= b_{-a}(\zeta^\alpha_k)$.
	\item[(ii)] If $\nu=\sum_{k=1}^N \nu_k \delta_{\zeta_k^\alpha}$, then
	\[
	U_aD_\nu U_a^*= D_{\nu^a},
	\]
	where $\nu^a=\sum_{k=1}^N \nu_k \delta_{\eta_k^\alpha}$.
	
	\item[(iii)] We have
	\[
	U_a\Delta^\alpha_B U^*_a= D_{\tilde{\nu}_\alpha^a},
	\]
	where
	\[
	\tilde{\nu}_\alpha^a=\frac{1}{|b'_{-a}\circ b_a|} \mu^{B^a}_\alpha =|b'_{a}|\mu^{B^a}_\alpha.
	\]
	\end{itemize}
\end{lemma}

\begin{proof}
	For $g\in K_{B^a}$ we have
	\[
	\< g, U_a(k^B_{\zeta_k^\alpha})\>
	=\< U_a^*g, k^B_{\zeta_k^\alpha}\>
	=( U_a^*g ) ( \zeta_k^\alpha)
	= \sqrt{b'_{-a}(\zeta_k^\alpha)} g(\eta_k^\alpha),
	\]
	which proves (i).
	
	To prove (ii), we use (i) to check the action of the left hand side operator on the reproducing kernels. Thus
	\[
	U_aD_\nu U_a^*(k^{B^a}_{\eta^\alpha_k})
	= \sqrt{b'_{a}(\eta_k^\alpha)} U_a D_\nu (k^B_{\zeta_k^\alpha})
	=\nu_k  \sqrt{b'_{a}(\eta_k^\alpha)} \sqrt{b'_{-a}(\zeta_k^\alpha)} k^{B^a}_{\eta^\alpha_k}
	=\nu_k k^{B^a}_{\eta^\alpha_k}.
	\]
	
	(iii) is a consequence of (ii): if $\nu=\mu^B_\alpha$, then
	\[
	\begin{split}
	\nu^a&=\sum_{k=1}^N \frac{1}{|B'(\zeta_k^\alpha)|} \delta_{\eta_k^\alpha}
	=\sum_{k=1}^N \frac{1}{|(B^a)'(\eta_k^\alpha)| \cdot  | b'_{-a}(\zeta_k^\alpha) }  \delta_{\eta_k^\alpha}\\
	&=\sum_{k=1}^N \frac{1}{|(B^a)'(\eta_k^\alpha)|| \cdot  | (b'_{-a}\circ b_a(\eta_k^\alpha) |}  \delta_{\eta_k^\alpha}
	=\frac{1}{|b'_{-a}\circ b_a|} \mu^{B^a}_\alpha=|b'_{a}|\mu^{B^a}_\alpha.\qedhere
	\end{split}
		\]
\end{proof}

\section{Sedlock algebras}\label{se:circulants}
Sedlock algebras have been introduced in~\cite{Se}, where it was shown that they are the only algebras contained in $\TT_B$.  The family of  Sedlock algebras is indexed by a parameter $\alpha\in\bbC\cup\{ \infty \}$. We will be interested only in $\alpha\in\bbT$, in which case the Sedlock algebra  $\BB^B_\alpha$ is  defined to be the commutant of the unitary operator $S^\alpha_B$.

In ~\cite[Proposition 3.2]{Se}  Sedlock characterizes the algebra $\BB^B_\alpha$ as the set of truncated Toeplitz operators with symbols of the
form
$$ \phi + \alpha\overline{S_B ( B \overline{z\phi}} )= \phi + \alpha\overline{P_B (B{\overline{\phi}}})   , $$
where $ \phi \in K_B$.
If $ B(0) = 0 $ we
have $P_B (B \overline{\phi}) = B(\overline{\phi - \phi(0))}$, and thus Sedlock's result yields the following Lemma.

\begin{lemma}\label{le:sedlock1}
	Suppose $B(0)=0$. A bounded operator $A$ is in $\BB^B_\alpha$ if and only if $A=T_B[\phi+\alpha\bar B (\phi-\phi(0))]$ for some $\phi\in K_B$. 
	\end{lemma}

Using Theorem~\ref{th:clarkpoltoratski}, one can connect the two descriptions of the Sedlock algebras, as shown by the next theorem.

 \begin{theorem}\label{th:circulants as functions generalized} Let $B$ be an arbitrary inner function satisfying $B(0)=0$. Suppose $T=T_B[\phi+\alpha\bar B (\phi-\phi(0))]\in \BB_\alpha^B$. Then the function $\phi$ has radial limits almost everywhere with respect to $\mu_\alpha$. If we denote the limit function by $\phi^*$, then $\phi^*\in L^\infty(\mu_\alpha)$, and $T=\phi^*(S_B^\alpha)$.
 \end{theorem}

 \begin{proof} The first part of the statement follows from Theorem~\ref{th:clarkpoltoratski}, since $\phi\in K_B$.
 	The operator $V:K_B\to L^2(\mu_\alpha)$ defined by $Vf=f^*$ is unitary, and $VS_B^\alpha=M_z V$. Since $\BB^B_\alpha=\{ S^\alpha_B \}'$, we have $VTV^*\in \{M_z\}'$, and thus $VTV^*=\Phi(M_z)=M_{\Phi}$ for some $\Phi\in L^\infty(\mu_\alpha)$.
 	
 	On the other hand, it follows from Proposition 3.2 of~\cite{Se}  that, when we write an operator $T\in \BB_\alpha$ as $T=T_B[\phi+\alpha\bar B (\phi-\phi(0))]$, we can identify $\phi$ as  $\phi=T\mathbf{1}$. So
 	\[
 	\phi^*=V\phi=VT\1=M_\Phi V\1=M_\Phi\1=\Phi.
 	\]
 	Thus
 	\[
 	T=V^*\phi^*(M_z)V=\phi^*(S^\alpha_B),
 	\]
 	which ends the proof of the theorem.
 \end{proof}

Some consequences for operators in the Sedlock class $\BB^B_\alpha$ can be immediately deduced.

\begin{corollary}\label{co:consequences of circ as functions}
	With the above notations, if we decompose $\mu_\alpha=\mu_\alpha^c+\mu_\alpha^a$ in its continuous and atomic parts, and the support of $\mu_\alpha^a$ is the sequence $(\eta_n)_{n}$, then:

\begin{itemize}
	\item[(i)] $T_B[\phi+\alpha\bar B (\phi-\phi(0))]$ is compact if and only if $\phi^*=0$ $\mu_\alpha^c$-almost everywhere, while $\phi^*(\eta_j)\to 0$.
	
	\item[(ii)] If $p$ is a positive real number, $T_B[\phi+\alpha\bar B (\phi-\phi(0))]$ is in the class $\CC_p$  if and only if $\phi^*=0$ $\mu_\alpha^c$-almost everywhere and $\sum_n |\phi^*(\eta_n)|^p<\infty$.
	
	\item[(iii)] If $p\ge 1$ is an integer and  $T=T_B[\phi+\bar B (\phi-\phi(0))]\in \CC_p$, then
	\[
	\Tr T^p=\sum_n  \phi^*(\eta_n)^p.
	\]
	
\end{itemize}
\end{corollary}

\section{Sedlock algebras and finite Blaschke products}\label{se:sedlock and finite blaschke}

Let us suppose now that $B$  is a finite Blaschke product such that $B(0)=0$. Then, for $\alpha\in\bbT$, the measure $\mu_\alpha^B$ is given by~\eqref{eq:definition of Clark measure}.   As noted above, for each $k$, $k^B_{\zeta_k^\alpha}$ is an eigenvector of $S_B^\alpha$ associated to the eigenvalue $\zeta_k^\alpha$. Therefore,  if $\phi\in K_B$,  Theorem~\ref{th:circulants as functions generalized} implies that
\begin{equation}\label{eq:decomposition of circulant}
T_B[\phi+\alpha\bar B (\phi-\phi(0))]=\sum_{k=1}^N \phi(\zeta_k^\alpha)\left(
\frac{k^B_{\zeta_k^\alpha}}{\|k^B_{\zeta_k^\alpha}\|} \otimes
\frac{k^B_{\zeta_k^\alpha}}{\|k^B_{\zeta_k^\alpha}\|}  \right),
\end{equation}

 If $T=T_B[\phi+\bar B (\phi-\phi(0))]$ and $\nu=\sum_{k=1}^N \nu_k \delta_{\zeta_k^\alpha}$,  it follows from~\eqref{eq:decomposition of circulant} and~\eqref{eq:definition of D_nu} that
\begin{equation}\label{eq:trace for general nu}
\Tr(D_\nu T^p)=\sum_{k=1}^N  \nu_k \phi^p(\zeta_k^\alpha)= \int_{\Bbb{T}} \phi^p\, d\nu. 
\end{equation}

In particular,
\begin{equation}\label{eq:trace for Clark measure in K_B}
\Tr(\Delta^\alpha_B T^p)= \int_{\Bbb{T}} \phi^p\, d\mu^B_\alpha
\end{equation}
and therefore, using also~\eqref{eq:definition of mualpha},
\begin{equation}\label{eq:trace norm for clark measure}
\|\Delta^\alpha_B\|_1=\Tr(\Delta^\alpha_B)= \int_{\Bbb{T}} d\mu^B_\alpha=1.
\end{equation}

We may integrate with respect to $\alpha$ to obtain some interesting consequences.

\begin{lemma}\label{le:integration of Delta wit respect to alpha}
	With the above notations, we have
	\[
	\int_\bbT \Delta^\alpha_{B} dm(\alpha)=  T_{B}\left[\frac{1}{|B'|}\right].
	\]
\end{lemma}

\begin{proof}
	For $f\in K_{B}$, we have $f=\sum_{k=1}^N \left\<f ,
	\frac{k^B_{\zeta_k^\alpha}}{\|k^B_{\zeta_k^\alpha}\|} \right\>
	\frac{k^B_{\zeta_k^\alpha}}{\|k^B_{\zeta_k^\alpha}\|}$. By (\ref{eq:aleksandrov}), we get
	\[
	\begin{split}
	\int_\bbT \Delta^\alpha_{B}(f) dm(\alpha)
	&=\int_\bbT \left( \sum_{k=1}^N \frac{1}{|B'(\zeta_k^\alpha)|}
	\left\<f , \frac{k^B_{\zeta_k^\alpha}}{\|k^B_{\zeta_k^\alpha}\|} \right\>
	\frac{k^B_{\zeta_k^\alpha}}{\|k^B_{\zeta_k^\alpha}\|}\right) dm(\alpha)\\
	&=\int_\bbT  \left(  \int_\bbT  \left\<f ,
	\frac{k^B_{\zeta}}{\|k^B_{\zeta}\|} \right\>
	\frac{k^B_{\zeta}}{\|k^B_{\zeta}\|}
	d\mu^{B}_\alpha (\zeta)  \right) dm(\alpha)\\
	&= \int_\bbT \frac{1}{|B'(\zeta)|} f(\zeta) k^B_{\zeta} dm(\zeta)\\
	&=T_{B}\left[\frac{1}{|B'|}\right] (f).\qedhere
	\end{split}
	\]
\end{proof}

From~\eqref{eq:trace for Clark measure in K_B} and Lemma~\ref{le:integration of Delta wit respect to alpha} follows the next corollary.

\begin{corollary}\label{co:trace after integration in K_B}
	With the above notations (including $T=T_B[\phi+\alpha\bar B (\phi-\phi(0))]$), we have
	\[
	\Tr\left( T_{B}\left[\frac{1}{|B'|}\right]T^p\right) =
	 \int \phi^p\, dm.
	\]
\end{corollary}



 \section{Approximating by circulants}\label{se:approx by circulants}

In this section we will develop a procedure, in the case $B(0)=0$, for deriving properties of a truncated Toeplitz operator from a closely associated element of a Sedlock algebra. The idea comes from the case of classical Toeplitz matrices (see, for instance,~\cite[Chapter 4]{Gr}) and it is based on an enlargement of the model space.

Fix   $\alpha \in \Bbb{T}$, and let $B$ be a finite Blaschke product with $B(0)=0$ and  $\psi$ be a (general) standard symbol for a truncated Toeplitz on $K_B$.  Thus $\psi \in K_B+\overline{K_B}$ is of the form $\psi=\psi_+ +\bar B \psi_-$ with $\psi_+\in K_B$ and $\psi_-\in zK_{\hat B}$.
Let $u$ be an inner function that is a multiple of $B$; that is, $u=Bv$ for some inner function $v$. Obviously, $\psi$ is also a standard symbol for a truncated Toeplitz operator on $K_u$ (since $K_B+\overline{K_B}\subset K_u+\overline{K_u}$). We intend to show that, if we set
\[
\phi=\psi_+  + {\overline{\alpha} }v\psi_-
\]
then we can have good estimates for the difference between $T_u[\psi]$ and the operator $T_u[\phi+\alpha {\bar{u}}(\phi-\phi(0))]$ in $\BB^u_\alpha$.  Notice that $\phi\in K_B \subset  K_u$
 and
that $\phi(0)=\psi_+(0)$ (since $\psi_-(0)=0$).

We see that:
\[
\begin{split}
\phi+ \alpha \bar u(\phi-\phi(0))&
=\psi_+ + \bar \alpha v\psi_- + \alpha \bar u( \psi_+ + \bar \alpha  v\psi_- - \psi_+(0) ) \\
&= \psi_++  \bar B\psi_- + \bar \alpha v\psi_- + \alpha \bar u (\psi_+  - \psi_+(0))\\
&= \psi + \bar \alpha v\psi_- + \alpha \bar u (\psi_+  - \psi_+(0))\\
\end{split}
\]
and thus
\begin{equation}\label{eq:difference to circulant}
T_u[\phi+ \alpha \bar u(\phi-\phi(0))]- T_u[\psi]=
\bar \alpha T_u[v\psi_-]+ \alpha T_u[\bar u (\psi_+-\psi_+(0))].
\end{equation}

We can now show that the norms of the two operators on the right hand side of~\eqref{eq:difference to circulant} do not depend on which inner function $u$ we choose. We will interpret these operators as operators on all of $L_2$, rather than just on the subspace $K_u$.  We use throughout
the computations below the fact that $\psi, \psi_+, \psi_-$ are all bounded functions.

We use repeatedly that $K_u = K_B \oplus BK_v = K_v \oplus vK_B$  and that $P_{vK_B} = M_vP_BM_{\bar v}.$ This gives us that:
\[
T_u [v\psi_-] =  P_u M_{v\psi_-} P_u= P_u M_{v\psi_-} (P_B+P_{BK_v}).
\]
But, it is clear that $Im(M_{v\psi_-} P_{BK_v}) \subset u H^2$ and so $P_u M_{v\psi_-} P_{BK_v} = 0.$ Thus:
\[
T_u [v\psi_-] =   P_u M_{v\psi_-} P_B= (P_v+P_{vK_B}) M_{v\psi_-} P_B.
\]
Now we use the fact that $Im( M_{v\psi_-} P_B) \subseteq vH^2$ so that $P_v M_{v\psi_-} P_B = 0$ to obtain that:
\begin{equation}\label{eq:first corner of difference}
P_u M_{v\psi_-} P_u= P_{vK_B} M_{v\psi_-} P_B = M_v P_BM_{\psi_-}P_B.
\end{equation}
So $T_u[v\psi_-]$ is obtained by multiplying with the operator $T_B[\psi_-]$ by a unitary operator. In particular, the norm (uniform or Schatten-von Neumann) of $T_u[\psi_-]$ depends only on $B$ and $\psi$ and is independent of the choice of $v.$

In our analysis of the second operator, we simplify the notation by denoting $\psi_0=\psi_+-\psi_+(0)$ so that $\psi_0\in zK_{\hat B}$. We have
\[
P_u M_{\bar u \psi_0}P_u= P_u M_{\bar u \psi_0}(P_v+P_{vK_B}).
\]
Now, using the fact that $K_v \subset vH^2_-$  and that $\psi_0 \in K_B \subset BH^2_-$ we see that
$Im(M_{\bar u \psi_0} P_{v} )\subset {\bar u} B v H^2_- = H^2_-.$ So it is clear that $P_u M_{\bar u \psi_0} P_{v} = 0.$ Thus

\[
P_u M_{\bar u \psi_0}P_u=  P_u M_{\bar u \psi_0}P_{vK_B}.
\]
Next we notice that:
\[
P_u M_{\bar u \psi_0}P_{vK_B} = P_u M_{\bar u \psi_0} M_v P_B M_{\bar v} = P_u M_{\bar B \psi_0} P_B M_{\bar v}.
\]
Since $\psi_0 \in K_B$ we know that ${\bar B}  \psi_0 \subset H^2_-$ . But,  an element of $K_B$ is a function $h \in H^2$ which is orthogonal to $BH^2$ and when we multiply such a function by a (bounded) element of $H^2_-$ we get an element $h_1$ of $L^2$ which is orthogonal
to $BH^2$ and is therefore orthogonal to $BK_v.$ Thus
 we can conclude that
\begin{equation}\label{eq:second corner of difference}
P_u M_{\bar u \psi_0}P_u = (P_{BK_v } + P_B )M_{\bar B \psi_0} P_B M_{\bar v}=P_B M_{\bar B\psi_0}P_B M_{\bar v}.
\end{equation}
So, just as in the first case, we have obtained that any norm (uniform or Schatten-von Neumann) of $T_u[\bar u (\psi_+-\psi_+(0))]$ is independent of the choice of $v$. We summarize the conclusion of our argument in the following lemma.

\begin{lemma}\label{le:approx by circulant}
	Suppose $\alpha\in\bbT$,  $\psi= \psi_++\bar B \psi_-\in K_B+\overline{K_B}$. If $u=Bv$ for some inner function $v$, and $\phi=\psi_+ + {\bar \alpha} v\psi_-$, then $T_u[\phi+ \alpha {\bar u}(\phi-\phi(0))]\in\BB^u_\alpha$ and, for any $0<p\le \infty$, we have
	\[\| T_u[\phi+ \alpha {\bar u}(\phi-\phi(0))]-T_u[\psi]\|_p\leq C_p,\]
	where $C_p$ is a finite constant independent of $v$.
\end{lemma}

\section{A Szeg\"o theorem for non-Blaschke sequences}\label{se:main theorem}

Suppose now that $\lambda_j\in \bbD$, $j\ge 1$. As in Section ~\ref{sse:clark measures}, define, for $n\ge 1$, $B_n=\prod_{j=1}^{n} (-\frac{|\lambda_j|}{\lambda_j}b_{\lambda_j})$.

The next theorem may be considered the central result of this paper; it gives a Szeg\"o theorem for truncated Toeplitz operators.

\begin{theorem}\label{th:general convergence for non-Blaschke}
	Suppose $\sum (1-|\lambda_j|)=\infty$, $B_n$ is defined as above, $\psi\in C(\bbT)$, and $p\in\bbN$. Then
		 \begin{equation}\label{eq:first general szego formula}
		 \Tr\left( T_{B_n}\left[\frac{1}{|B_n'|}\right](T_{B_n}[\psi])^p\right)\to \int_\bbT \psi^p\ dm.
		 \end{equation}
		
		 If $\psi$ is real-valued, then for every continuous function $g$ on $[ \inf\psi, \sup \psi  ]$ we have
		\begin{equation}\label{eq:second general szego formula}
		\Tr\left( T_{B_n}\left[\frac{1}{|B_n'|}\right]g(T_{B_n}[\psi])\right)\to \int_\bbT g\circ\psi\ dm.
		\end{equation}
\end{theorem}

The proof of the theorem will use a series of intermediate results which have interest in themselves. The basic technique is introduced in the next lemma.

\begin{lemma}\label{le:general convergence}
	Fix $N\in\bbN$, $\alpha\in\bbT$, assume $B_N(0)=0$, and denote by $\zeta_{n,k}^\alpha$, $k=1,\dots, n$, the roots of the equation $B_n(z)=\alpha$. Suppose that for each $n\ge N$  we are given
	an atomic measure $\nu_n=\sum_{k=1}^{n}\nu^{n}_k \delta_{\zeta_{n,k}^\alpha}$; denote by $D_{\nu_n}$ the corresponding diagonal operator as defined by~\eqref{eq:definition of D_nu}. Assume that
	\begin{itemize}
		\item[(i)] $\nu_n$ is weakly convergent to some measure $\nu$.
		\item[(ii)] $\|D_{\nu_n}\|\to 0$.		
	\end{itemize}
	Then for any $\psi\in K_{B_N}+\overline{K_{B_N}}$ and $p\in\bbN$ we have
	\[
	\Tr(D_{\nu_n}(T_{B_n}[\psi])^p)\to\int \psi^p\, d\nu.
	\]
\end{lemma}

\begin{proof}
	 If $\psi= \psi_+ +\bar B_N \psi_-\in K_{B_N}+\overline{K_{B_N}}$, we define $\phi_n=\psi_+ +\bar{\alpha} \frac{B_n}{B_N}\psi_-$ and $\tilde{\phi}_n= \phi_n+\alpha\bar B_n(\phi_n-\phi_n(0))$ for $n\geq N$. Applying Lemma~\ref{le:approx by circulant} for $B=B_N$ and $u=B_n$ we obtain that the trace norm $\|T_{B_n}[\psi]- T_{B_n}[\tilde\phi_n]\|_1$ is bounded by a constant independent of $n$. Therefore, using condition (ii),
	\[
	\begin{split}
	&|\Tr (D_{\nu_n}(T_{B_n}[\psi])^p) - \Tr(D_{\nu_n}(T_{B_n}[\tilde\phi_n])^p)| \\
	&\qquad  \leq
	\sum_{k=0}^{p-1}
	\left|\Tr \left( D_{\nu_n} \left(T_{B_n}[\psi]^k \left( T_{B_n}[\psi]-T_{B_n}[\tilde\phi_n]\right) T_{B_n}[\tilde\phi_n]^{p-k-1} \right)\right)\right| \\
	&\qquad\le M \|D_{\nu_n}\|\cdot
	\|T_{B_n}[\psi]- T_{B_n}[\tilde\phi_n]\|_1\to 0.
	\end{split}
	\]
	
	Now, according to~\eqref{eq:trace for general nu}, we have
	\[
	\Tr(D_{\nu_n}T_{B_n}[\tilde\phi_n]^p)=\int \phi_n^p\,d\nu_n.
	\]	
	Since $B_n(\zeta^{\alpha}_{n,k})=\alpha$, we have $\psi(\zeta^{(n)}_j)=\phi_n(\zeta^{(n)}_j)$ for all $j=1,\dots, n$, and thus the last integral is $\int \psi\, d\nu_n$, which by (i) converges to $\int \psi\, d\nu$.
\end{proof}

Next, with a suitable assumption on the symbol, we eliminate the restriction $B(0)=0$.

\begin{lemma}\label{le:convergence with alpha fixed, psi restricted }
	Let $\sum (1-|\lambda_j|)=\infty$, $B_n$ be as above, $\alpha\in\bbT$ fixed, and $\Delta^\alpha_{B_n}$ as defined by~\eqref{eq:definition of Delta}.  Suppose $\psi\in (K_{B_N^{-\lambda_1}}+\overline{K_{B_N^{-\lambda_1}}})\circ b_{\lambda_1}$ and $p\in\bbN$. Then
		\begin{equation}\label{eq:general convergence of traces with nu}
			\Tr(\Delta^\alpha_{B_n}(T_{B_n}[\psi])^p)\to\int \psi^p\, dm.
		\end{equation}
\end{lemma}

\begin{proof}
	Denote $a=-\lambda_1$, so $B_N(-a)=0$, Then for any $n\ge N$ we have
	\[
	B_n^a(0)=B_n\circ b_a(0)=B_n(-a) =0.
	\]
	
	We have, by Lemma~\ref{le:CGRW} and Lemma~\ref{le:action of U_a on kernels and measures} (iii)
	\[
	\begin{split}
	\Tr(\Delta^\alpha_{B_n}(T_{B_n}[\psi])^p)
	&= 	\Tr((U_a\Delta^\alpha_{B_n}U_a^* )(U_aT_{B_n}[\psi]U_a^*)^p)\\
	&= \Tr (D_{\tilde{\nu}_\alpha^{a, n} }(T_{B_n^a}[\psi\circ b_a] )^p) ,	\end{split}
	\]
	where
	\[
	\tilde{\nu}_\alpha^{a,n}=\frac{1}{|b'_{-a}\circ b_a|} \mu^{B_n^a}_\alpha.
	\]
	
	Consider the sequence of measures $	\tilde{\nu}_\alpha^{a,n}$. By Lemma~\ref{le:convergence of clark measures} (i),
	\[
	\tilde{\nu}_\alpha^{a,n}\to \frac{1}{|b'_{-a}\circ b_a|} m.
	\]
	On the other hand,
	\[
	\|  D_{\tilde{\nu}_\alpha^{a, n} }\|
	=\| \Delta^\alpha_{B_n}\|=\sup_{1\le j\le n} \frac{1}{|B_n'(\zeta^{(n)}_j)|}\le
	\sup_{\zeta\in\bbT} \frac{1}{|B_n'(\zeta)|}\le \frac{2}{\sum_{j=1}^n(1-|\lambda_j|)},
	\]
	where the last inequality is a consequence of Lemma~\ref{le:blaschke argument derivative} (ii).
	
	The assumption on $\psi$ implies that $\psi\circ b_a\in K_{B_N^{a}}+\overline{K_{B_N^{a}}}$. Therefore, applying 
	Lemma~\ref{le:general convergence}, 	\[
	\Tr(\Delta^\alpha_{B_n}(T_{B_n}[\psi])^p)
	= \Tr (D_{\tilde{\nu}_\alpha^{a, n} }(T_{B_n^a}[\psi\circ b_a] )^p) \to \int \frac{1}{|b'_{-a}\circ b_a|} (\psi\circ b_a)^p\, dm. 	\]
	By Lemma~\ref{le:change of Clark by automorphism}, the last integral is precisely $\int \psi^p\, dm$, which finishes the proof.
\end{proof}

The next step is to integrate with respect to $\alpha\in\bbT$ in order to obtain a version of~\eqref{eq:first general szego formula} for a restricted class of functions~$\psi$.

\begin{lemma}\label{le:integrated general convergence}
	Let $\sum (1-|\lambda_j|)=\infty$, $B_n$ be as above,     $\psi\in (K_{B_N^{-\lambda_1}}+\overline{K_{B_N^{-\lambda_1}}})\circ b_{\lambda_1}$ , and $p\in\bbN$.
	Then
	\[
	\Tr\left( T_{B_n}\left[\frac{1}{|B_n'|}\right](T_{B_n}[\psi])^p\right)\to \int_\bbT \psi^p\ dm.
	\]
	
\end{lemma}

\begin{proof}
	By Lemma~\ref{le:convergence with alpha fixed, psi restricted } we have for each $\alpha\in\bbT$
		\begin{equation}\label{eq:th 5.2}
		\Tr(\Delta^\alpha_{B_n}(T_{B_n}[\psi])^p)\to\int \psi^p\, dm.
		\end{equation}

On the other hand, 	from Lemma~\ref{le:integration of Delta wit respect to alpha} it follows that
		\[
		\begin{split}
		\int_\bbT \Tr(\Delta^\alpha_{B_n}(T_{B_n}[\psi])^p)dm(\alpha)&=\Tr\left(\left(\int_\bbT \Delta^\alpha_{B_n} dm(\alpha)\right)  (T_{B_n}[\psi])^p\right) \\
		&=\Tr\left( T_{B_n}\left[\frac{1}{|B_n'|}\right](T_{B_n}[\psi])^p\right).
		\end{split}
		\]
		Then~\eqref{eq:th 5.2} together with Lebesgue's dominated convergence theorem imply that
		\[
		\Tr\left( T_{B_n}\left[\frac{1}{|B_n'|}\right](T_{B_n}[\psi])^p\right)\to \int_\bbT\left( \int_\bbT \psi^p\,  dm\right) dm(\alpha)=\int_\bbT \psi^p\ dm. \qedhere
		\]
		\end{proof}

We may now go to the proof of the main theorem.

\begin{proof}[Proof of Theorem~\ref{th:general convergence for non-Blaschke}]
Suppose now $\psi\in C(\bbT)$ and $p\in\bbN$.
	Let $\epsilon>0$ be given.  Corollary~\ref{co:desity fo kernels} implies that $\bigvee_n (K_{B_N^{-\lambda_1}}+\overline{K_{B_N^{-\lambda_1}}})\circ b_{\lambda_1}$ is also dense in $C(\bbT)$. Next we choose 
 $N\in\bbN$ and $\psi_N\in (K_{B_N^{-\lambda_1}}+\overline{K_{B_N^{-\lambda_1}}})\circ b_{\lambda_1}$
	such that $\|\psi-\psi_N\|_{\infty}\le\epsilon$, where  $\|.\|_\infty$ is the supremum norm; we assume also that $\|\psi\|_\infty, \|\psi_N\|_\infty\le M$. 
	
	Applying now Lemma~\ref{le:integrated general convergence} to $\psi_N$, we see that there exists $n_0$ such that for $n\ge n_0$ we have
	\[
	\Big|	\Tr (T_{B_n}\left[\frac{1}{|B_n'|}\right]((T_{B_n}[\psi_N])^p)-\int \psi_N^p\,dm\Big|\le \epsilon.
	\]
	Also, since $\|T_{B_n}[\phi]\|\le \|\phi\|_\infty$ for any $\phi\in C(\bbT)$, we have
	\[
	\begin{split}
	\|T_{B_n}[\psi_N])^p- T_{B_n}[\psi])^p\|&
	\le \sum_{k=0}^{p-1} \|T_{B_n}[\psi_N]^k\| \|T_{B_n}[\psi]^{p-k-1} \|
	\|T_{B_n}[\psi_N]- T_{B_n}[\psi]\|\\
	&\le pM^{p-1}\epsilon.
	\end{split}
	\]
	Using~\eqref{eq:trace norm for clark measure}, we have
\[	
\|	\Delta^\alpha_{B_n}\|_1=
	\Tr \Delta^\alpha_{B_n} =\Re \frac{\alpha+B_n(0)}{\alpha-B_n(0)}\le \frac{1+|B_n(0)|}{1-|B(0)|}\le \frac{1+|\lambda_1|}{1-|\lambda_1|}.
\]	
Therefore
\begin{equation}\label{eq:again trace}
\left\Vert T_{B_n}\left[\frac{1}{|B_n'|}\right]\right\Vert_1 \le \int_\bbT \Vert \Delta^\alpha_{B_n}\Vert_1 dm(\alpha) \le  \frac{1+|\lambda_1|}{1-|\lambda_1|},
\end{equation}
and
	\[
	\begin{split}
	&\left|\Tr(T_{B_n}\left[\frac{1}{|B_n'|}\right]((T_{B_n}[\psi])^p)- \int\psi^p\, dm \right|\\
	&\qquad \le
	\Tr \left( T_{B_n}\left[\frac{1}{|B_n'|}\right](T_{B_n}[\psi])^p-T_{B_n}\left[\frac{1}{|B_n'|}\right](T_{B_n}[\psi_N])^p   \right)
	\\
	&\qquad\qquad +\Big|	\Tr (T_{B_n}\left[\frac{1}{|B_n'|}\right](T_{B_n}[\psi_N])^p-\int \psi_N^p\, dm\Big|	+ \int |\psi_N^p - \psi^p|\, dm
	\\
	&\qquad \le pM^{p-1}\frac{1+|\lambda_1|}{1-|\lambda_1|}\epsilon+\epsilon+ pM^{p-1}\epsilon.
	\end{split}
	\]
	The above estimate proves formula~\eqref{eq:first general szego formula}.		
	
	Suppose now that $\psi$ is real valued and $g$ is continuous on $[\inf \psi, \sup \psi]$. Note that  $T_{B_n}[\psi]$ is then a bounded self-adjoint operator with spectrum contained in $[\inf \psi, \sup \psi]$. Let $(P_k)_{k\in\bbN}$ be a sequence of polynomials that converges to $g$ uniformly on $[\inf \psi, \sup \psi]$.

	By~\eqref{eq:first general szego formula}, for every $k$ we have
	\[
	\Tr\left( T_{B_n}\left[\frac{1}{|B_n'|}\right]P_k(T_{B_n}[\psi])\right)\to \int_\bbT P_k\circ\psi\ dm.
	\]
	But
	\[
	\begin{split}
	&	\left|
		\Tr\left( T_{B_n}\left[\frac{1}{|B_n'|}\right]P_k(T_{B_n}[\psi])\right)-\Tr\left( T_{B_n}\left[\frac{1}{|B_n'|}\right]g(T_{B_n}[\psi])\right)
		\right|\\
	&\qquad	\le \left\Vert T_{B_n}\left[\frac{1}{|B_n'|}\right]\right\Vert_1  \Vert (P_k-g)(T_{B_n}[\psi])\Vert\\
	&\qquad  \le \frac{1+|\lambda_1|}{1-|\lambda_1|} \|P_k-g\|_\infty
	\end{split}
	\]
	where on the last line we have used~\eqref{eq:again trace}.
	
	Since obviously
	\[
		\left| \int_\bbT P_k\circ\psi-g\circ\psi\ dm \right|\le \|P_k-g\|_\infty,
	\]
	letting $k\to\infty$ proves~\eqref{eq:second general szego formula}.
\end{proof}

In the classical case of Toeplitz matrices, we have $B_n=z^n$, $|B_n'|\equiv n $, and $T_{B_n}\left[\frac{1}{|B_n'|}\right]=\frac{1}{n}I$. We obtain therefore a version of a classical Szeg\"o limit theorem: if $T_n$ are truncated Toeplitz matrices corresponding to the symbol $a$, then
\[
\frac{1}{n} \Tr T_n^p\to \int_{\bbT} a^n\, dm.
\]

It is interesting to compare Theorem~\ref{th:general convergence for non-Blaschke} with the main result in~\cite{B}, where a formula is obtained  for the asymptotics of $\det T_{B_n}[\psi]$. In the classical case the determinant and the trace versions of the Szeg\"o limit theorem are closely connected, and one can deduce one version from the other. This seems not to be the case for general truncated Toeplitz operators.

\section{Blaschke sequences}\label{blaschke sequences}

Theorem~\ref{th:general convergence for non-Blaschke} concerns sequences that do not 
satisfy the Blaschke condition. It is possible to obtain Szeg\"o results for Blaschke 
sequences, but these are less elegant and require supplementary hypotheses. First, Lemma~\ref{le:convergence with alpha fixed, psi restricted } is true for a Blaschke sequence, provided we add the condition $\|\Delta^\alpha_{B_n}\|\to 0$. Similarly, Lemma~\ref{le:integrated general convergence} is true if we assume that $\|\Delta^\alpha_{B_n}\|\to 0$ for almost all $\alpha\in\bbT$. As for the main Theorem~\ref{th:general convergence for non-Blaschke}, a new complication arises, since Lemma~\ref{le:density of kernels} is no longer true.. We do, however, 
have the following precise result;  recall that the space $\mathfrak{L}_\Lambda$ is defined in section 1, before Lemma~\ref{le:density of kernels}.

\begin{theorem}\label{th:convergence for Blaschke sequence } 
Suppose that $(\lambda_j)$ is a Blaschke sequence (i.e. that $\sum_{j=1}^{\infty}(1-|\lambda_j|)<\infty $) and each $\lambda_j$  is repeated $m_j$ times, with $m_j$ a finite integer. 
Let  $B_n$ be given by~(\ref{eq:definition of Bn}) and $p\in\bbN$.
Suppose that, for some fixed $i\ge 1$,   $ \psi\circ b_{-\lambda_i}\in\mathfrak{L}_\Lambda$.

	(1) If  $\alpha\in\bbT$  and  if   $\|\Delta^\alpha_{B_n}\|\to 0$,  then
	\begin{equation}\label{eq:general convergence of traces}
	\Tr(\Delta^\alpha_{B_n}(T_{B_n}[\psi])^p)\to\int \psi^p\, d\mu^B_\alpha.
	\end{equation}
	
	(2) If $\|\Delta^\alpha_{B_n}\|\to 0$ for almost all $\alpha\in\bbT$, then
	\[
	\Tr\left( T_{B_n}\left[\frac{1}{|B_n'|}\right](T_{B_n}[\psi])^p\right)\to \int_\bbT \psi^p\ dm.
	\]
\end{theorem}

Unfortunately, it does not seem possible to give simple conditions that would characterize those sequences for which
\begin{equation}\label{eq:|Delta|to zero}
\|\Delta^\alpha_{B_n}\|\to 0.
\end{equation}
The next two examples suggest the intricacy of the problem. First we give an example of a Blaschke sequence for which~\eqref{eq:|Delta|to zero} is satisfied.

\begin{example}
	Suppose $(\lambda_j)$ is a sequence of points in $\bbD$ obtained by choosing 
$m^2$ evenly distributed points on each circle of radius $r_m=1-\frac{1}{m^4}.$ We have then
	\[
	\sum(1-|\lambda_j|)=\sum_m m^2\cdot \frac{1}{m^4} <\infty,
	\]
	and therefore the sequence is Blaschke.
	
	Fix $m\in\bbN$. For any $e^{it}\in \bbT$ there exists a $j\in \bbN$ such that $|\lambda_j|=r_m$ and $|e^{it}-\lambda_j|\le \frac{10}{m^2}$;  therefore, using Lemma~\ref{le:blaschke argument derivative}
	\[
	P_{\lambda_j}(e^{it})=\frac{1-|\lambda_j|^2}{|e^{it}-\lambda_j|^2}\ge \frac{1/m^4}{100/m^4}=\frac{1}{100}.
	\]
	and so, for any $e^{it}\in \bbT$,	
	\[
	\sum_{|\lambda_j|=r_m} P_{\lambda_j}(e^{it})\ge \frac{1}{100} .
	\]
	If we fix $N$, and choose $n\ge \max\{j\in \bbN: |\lambda_j|=r_m, 0\le m\le N  \}$, then
	\[
	\inf\{|B_n'(\zeta): \zeta\in \bbT  \}\ge N/100.
	\]
	So in this case $\inf\{|B_n'(\zeta)|: \zeta\in\bbT\}\to \infty$ when $n\to\infty$, and thence $\|\Delta^\alpha_{B_n}\|\to 0$, for any $\alpha\in\Bbb{T}$.

\end{example}

\begin{example}

The second example exhibits the opposite behavior. We show that, for certain Blaschke sequences $(\lambda_j )$ the hypothesis ~\eqref{eq:|Delta|to 
zero} is not satisfied, and, in the case where $\alpha=1, \lambda_1 = 0$ 
and $\psi = \bar{z},$ the conclusion~\eqref{eq:general convergence of 
traces} is also false; that is: 
\[
\Tr (\Delta^1_{B_n}T_{B_n}[\psi])\not\to \int \psi \, d\mu_1^B.
\]

So, let $(\lambda_j )$ be a real valued Blaschke sequence with 
$\lambda_1 = 0.$ We will add other hypotheses as necessary. Following the approximation procedure of Section~\ref{se:approx by circulants} for the finite Blaschke product
$ B_2(z)=-z(-b_{\lambda_2}(z))$ (where $B_n$  is defined 
by~(\ref{eq:definition of Bn}) and the remark after it), we set $N=2$, $\psi (z) = \bar z$ and 
$v = \prod_{j=3}^{n} (-b_{\lambda_j}).$ Then  $u=B_n=B_2 v$, 
 $\psi_+=0$, $\psi_-=b_{\lambda_2}(z)$,  and $\phi=B_n/z$. Formula~\eqref{eq:difference to circulant} then becomes
\[
T_{B_n}[B_n/z+\bar B_n(B_n/z- B_n/z(0))]-T_{B_n}[\bar z]= T_{B_n}[B_n/z].
\]

From~\eqref{eq:trace for Clark measure in K_B} and the fact that $B_n\equiv 1$ on the support of $\mu^{B_n}_1$, we obtain
\[
\Tr \left(\Delta^{1}_{B_n} T_{B_n}[B_n/z+\bar B_n(B_n/z- (B_n/z)(0))]\right)=\int_\Bbb{T} B_n/z \, d\mu_1^{B_n}= \int_\Bbb{T} \bar z \, d\mu_1^{B_n}.
\]
Thus, for $n\to\infty$,
\[
\Tr \left(\Delta^{1}_{B_n} T_{B_n}[B_n/z+\bar B_n(B_n/z- (B_n/z)(0))]\right) \to \int_\Bbb{T} \bar z \, d\mu^B_1.
\]

So, in order to achieve our purpose of  showing that
\[
\Tr (\Delta^1_{B_n}T_{B_n}[\bar z])\not\to \int_\Bbb{T} \bar z \, d\mu_1^B,
\]
we have to prove that
\begin{equation}\label{eq:what we have to prove}
\Tr (\Delta^1_{B_n}T_{B_n}[B_n/z])\not\to 0.
\end{equation}

We next use~\cite[Section 5]{Sa}, where it is shown that
\[
T_{B_n}[B_n/z]= \tilde k_0^{B_n}\otimes k^{B_n}_0.
\]

Denote by $\zeta^n_k$  the solutions to $B_n(\zeta)=1$. Notice that, since $-1$ is such a solution,  $\zeta^n_1=-1$. We have
\[
\Delta^1_{B_n} =\sum_{k=1}^{n} \frac{1}{|B'_n(\zeta^n_k)|}
\left( \frac{k^{B_n}_{\zeta^n_k}}{\| k^{B_n}_{\zeta^n_k} \|} \otimes \frac{k^{B_n}_{\zeta^n_k}}{\| k^{B_n}_{\zeta^n_k} \|} \right)
= \sum_{k=1}^{n} \frac{1}{|B'_n(\zeta^n_k)|^2} (k^{B_n}_{\zeta^n_k} \otimes k^{B_n}_{\zeta^n_k}  )
\]
(since $\| k^{B_n}_{\zeta^n_k} \|^2=k^{B_n}_{\zeta^n_k}(\zeta^n_k)= |B'_n(\zeta^n_k)| $ by~\eqref{eq:formula for norm of k_zeta}).

Using the usual formulas
\[
(x\otimes y)(z\otimes t)= \<z, y\> (x\otimes t),
\quad
\Tr (x\otimes y)= \<x, y\>,
\]
we obtain
\[
\Delta^1_{B_n}T_{B_n}[B_n/z] =
\sum_{k=1}^{n} \frac{1}{|B'_n(\zeta^n_k)|^2} \< \tilde k^{B_n}_0, k^{B_n}_{\zeta^n_k} \> (k^{B_n}_{\zeta^n_k}\otimes k^{B_n}_0)=
\sum_{k=1}^{n} \frac{\tilde k^{B_n}_0(\zeta^n_k)}{|B'_n(\zeta^n_k)|^2} (k^{B_n}_{\zeta^n_k}\otimes k^{B_n}_0),
\]
and
\[
\Tr\big(\Delta^1_{B_n}T_{B_n}[B_n/z]  \big)=
\sum_{k=1}^{n} \frac{\tilde k^{B_n}_0(\zeta^n_k)}{|B'_n(\zeta^n_k)|^2}
\< k^{B_n}_{\zeta^n_k}, k^{B_n}_0  \>=
\sum_{k=1}^{n} \frac{\tilde k^{B_n}_0(\zeta^n_k) \overline{k^{B_n}_0(\zeta^n_k)}}{|B'_n(\zeta^n_k)|^2}.
\]
We have $\tilde k^{B_n}_0=B_n/z$, so $\tilde k^{B_n}_0(\zeta^n_k)=B_n(\zeta^n_k)\bar \zeta^n_k=\bar \zeta^n_k$, while  $k^{B_n}_0\equiv 1$, so $k^{B_n}_0(\zeta^n_k)=1$. Thus we have arrived at the formula
\begin{equation}\label{eq:trace for the example}
\Tr\big(\Delta^1_{B_n}T_{B_n}[B_n/z]  \big)=
\sum_{k=1}^{n} \frac{\bar\zeta^n_k}{|B'_n(\zeta^n_k)|^2}=
-\frac{1}{|B'_n(-1)|^2}+\sum_{k=2}^{n} \frac{\bar\zeta^n_k}{|B'_n(\zeta^n_k)|^2}.
\end{equation}

Now we begin putting hypotheses on the rest of the sequence $(\lambda_j).$ We suppose 
that, $\lambda_j \nearrow 1$ and that $\lambda_j>1/2$ for $j\ge 2$, which implies that
\begin{equation*}
|B_n'(-1)|\le 1+\frac{4}{9} \sum_{j=2}^{n} (1-\lambda_j^2),
\end{equation*}
so, if we assume that the $\lambda_j$ tend sufficiently fast to 1, we may assume that $|B_n'(-1)|\le 3/2$.

To estimate the remaining sum, remember that by Lemma~\ref{le:blaschke argument derivative} (i)  we have
\begin{equation*}
|B'_n(\zeta)|=\sum_{j=1}^{n} \frac{1-\lambda_j^2}{|\zeta-\lambda_j|^2}=
1+\sum_{j=2}^{n} \frac{1-\lambda_j^2}{|\zeta-\lambda_j|^2}.
\end{equation*}
It follows that for fixed $\zeta$, $|B'_n(\zeta)|$ increases to $|B'(\zeta)|$, while for fixed $n$ it increases on both half circles of $\bbT$ as we are getting closer to 1.

Finally, we add the assumption that the sequence  $\lambda_j$ has been chosen so that $|B'(\zeta)|\le 11/10$ for all $\zeta$ belonging to the arc $J$ that connects $e^{i/10}$ and $e^{-i/10}$ and contains $-1$. Then, for $\zeta\in J$, $|B_n'(\zeta)|\le |B'(\zeta)|\le 11/10$, and so, by Lemma~\ref{le:blaschke argument derivative}, $\frac{d}{dt}\Arg B_n(e^{it})\le 11/10$.

%
%
%

  Since  $\frac{d}{dt}\Arg B_n(e^{it})\le
11/10$ for $t\in [1/10, \pi]$, we have
\[|\Arg B_n(-1)-\Arg B_n(e^{i/10})|\le 11/10(\pi-\theta)\le 11\pi/10.
\]

Suppose $\zeta^n_2=e^{i\eta_2}$ is the next zero of $B_n(\zeta)=1$ going from $-1$ towards 1 on the upper semicircle. Then $\eta_2\in (0, 1/10)$, and, since $|\Arg B_n(-1)-\Arg B_n(e^{i\eta_2})|=2\pi$, we must have
\[
|\Arg B_n(e^{i/10})-\Arg B_n(e^{i\eta_2})|\ge 9\pi/10.
\]
But, using Lagrange's formula, there is some point $\theta_0\in (\eta_2, 1/10)$ such that
\[
\begin{split}
|\Arg B_n(e^{i/10})-\Arg B_n(e^{i\eta_2}|&=(\frac{1}{10}-\eta_2) \frac{d}{dt} \Arg B_n(e^{i\theta_0})\\
&\le \frac{1}{10}|B_n'(e^{i\theta_0})|\le \frac{1}{10}  |B_n'(e^{i\eta_2})|,
\end{split}
\]
the last inequality being a consequence of $0<\eta_2<\theta_0$. Therefore
$|B_n'(e^{i\eta_2})|\ge 9\pi$.

A similar argument can be used on the lower semicircle. Finally,
 since   $|B_n'(\zeta)|$ increases as we are getting closer to 1, we may conclude that $|B'_n(\zeta^n_k)|\ge 9\pi$ for any $k\ge 2$.

Remember now that
\[
\sum_{k=1}^{n}\frac{1}{|B'_n(\zeta^n_k)|}=\|\mu^{B_n}_1\|=1.
\]
We have then, starting with~\eqref{eq:trace for the example},
\[
\begin{split}
|\Tr\big(\Delta^1_{B_n}T_{B_n}[B_n/z]  \big)|
&\ge
\frac{1}{|B'_n(-1)|^2}-\sum_{k=2}^{n} \frac{1}{|B'_n(\zeta^n_k)|^2}\\
&\ge \frac{4}{9}-\frac{1}{9\pi} \sum_{k=2}^{n} \frac{1}{|B'_n(\zeta^n_k)|}\\
&\ge \frac{4}{9}-\frac{1}{9\pi}>\frac{1}{3}.
\end{split}
\]
Therefore the inequality (\ref{eq:what we have to prove}) is proved, which concludes the example.
\end{example}

\end{document}